\documentclass[12pt,reqno]{amsart}
\usepackage[utf8]{inputenc}
\usepackage{graphicx}
\graphicspath{ {./images/} }

\usepackage{amsthm}
\usepackage{cleveref}
\theoremstyle{definition}
\usepackage{amsmath}
\usepackage{amsfonts,latexsym,amsthm,amssymb,amsmath,amscd,euscript}
\usepackage{caption}
\usepackage{subcaption}
\usepackage{floatrow}
\usepackage{url}

\numberwithin{equation}{section}
\newtheorem{theorem}{Theorem}[section]

\newtheorem{lemma}[theorem]{Lemma}

\newtheorem{definition}[theorem]{Definition}
\newtheorem{example}[theorem]{Example}
\newtheorem{rek}[theorem]{Remark}

\newcommand{\burl}[1]{\textcolor{blue}{\url{#1}}}

\title{Tinkering with Lattices: A New Take on the Erd\H{o}s Distance Problem}
\author[ Boldyriew, Kim,  Miller, Palsson,  Sovine,  Trejos Suárez, Zhao ]{Elzbieta Boldyriew, Elena Kim, Steven J. Miller, Eyvindur Palsson, Sean Sovine, Fernando Trejos Suárez, Jason Zhao}
\date{\today}

\thanks{ This research was supported NSF grant DMS1947438 and Williams College. The third named author was supported by  NSF grant DMS1561945, the fourth named author was supported by Simons Foundation Grant \#360560, the fifth named author was supported by NSF grant DMS1907435, the sixth named author was supported by Yale University, and finally the last named author was supported by University of California-Los Angeles}

\begin{document}

\maketitle

\begin{abstract}
The Erd\H{o}s distance problem concerns the least number of distinct distances that can be determined by $N$ points in the plane. The  integer lattice with $N$ points is known as \textit{near-optimal}, as it spans $\Theta(N/\sqrt{\log(N)})$ distinct distances, the lower bound for a set of $N$ points (Erd\H{o}s, 1946). The only previous non-asymptotic work related to the Erd\H{o}s distance problem that has been done was for $N \leq 13$. We take a new non-asymptotic approach to this problem in a model case, studying the distance distribution, or in other words, the plot of frequencies of each distance of the $N\times N$ integer lattice. In order to fully characterize this distribution,
 we adapt previous number-theoretic results from Fermat and Erd\H{o}s in order to relate the frequency of a given distance on the lattice to the sum-of-squares formula.

We study the distance distributions of all the lattice's possible subsets;  although this is a restricted case, the structure of the integer lattice allows for the existence of subsets which can be chosen so that their distance distributions have certain properties, such as emulating the distribution of randomly distributed sets of points for certain small subsets, or emulating that of the larger lattice itself. We define an error which compares the distance distribution of a subset with that of the full lattice. The structure of the integer lattice allows us to take subsets with certain geometric properties in order to maximize error; we show these geometric constructions explicitly. Further, we calculate explicit upper bounds for the error when the number of points in the subset is $4$, $5$, $9$ or $\left \lceil N^2/2\right\rceil$ and prove a lower bound in cases with a small number of points.
\end{abstract}

\tableofcontents

\section{Introduction}
In 1946, Paul Erd\H{o}s \cite{Er1} proposed the now-famous  Erd\H{o}s distinct distances problem: given $N$ points in a plane, what is the minimum number of distinct distances, $f(N)$,  they can determine?
He accompanied this question with the first bounds on $f(N)$, 
\begin{equation}\sqrt{N-\frac{3}{4}}-\frac{1}{2} \ \leq \ f(N) \ \leq \ \frac{cN}{\sqrt{\log N}},\end{equation} and further conjectured that the upper bound was tight---to this day, nobody has found evidence to contradict this conjecture. However, since 1946, incidence theory and algebraic geometry have provided a series of improvements on the original lower bound, culminating with Guth and Katz's seminal result in 2015 \cite{GK}, which proved a lower bound of $\Omega(N/\log N)$.

Since Erd\H{o}s's original upper bound, coming from an estimate for the number of distinct distances on the $\sqrt{N} \times \sqrt{N}$ integer lattice \cite{EF}, has not been improved on to this day, any set with $O(N/\sqrt{\log N})$ distinct distances is known as \emph{near-optimal}. Erd\H{o}s further conjectured in 1986 \cite{Er2} that any near-optimal set has a lattice structure, although the truth of this conjecture remains an open problem for large values of $N$. 

In addition to the  Erd\H{o}s distinct distances problem, a significant amount of work has been published on related problems which analyze aspects of distributions of distinct distances on planar point sets. The unit distance problem, for instance, focuses on the number of times a single given distance---often, the unit distance---can appear in a planar set of $N$ points. However, most of the work done on these subjects has been asymptotic; previous non-asymptotic work has only been conducted on $N \leq 13$ \cite{BMP}. 

In this paper, we take a novel approach and examine the whole distance distribution for the lattice and its subsets in a non-asymptotic setting.  Although working in $\mathbb{Z}^2$ is a simplification, our work is complicated by considering the whole distance distribution---namely, taking into account the frequency with which each distance appears on the lattice---rather than working asymptotically with only the number of distinct distances.
In particular, we  first examine the distance distribution for the lattice, characterizing its behavior and applying number-theoretic methods to determine an upper bound for the frequency of its most common distance, showing the asymptotic bound in Equation ~\eqref{upperbound}. We then turn to the distance distributions of subsets of the lattice and compare them to the distance distribution of the lattice itself.  Although the sets are subsets of a highly regular set, the behavior of the distance distributions for these sets can vary widely. Some subsets have distance distributions that highly mimic that of the full lattice, while others have distance distributions that are similar to that of a random set. We devise in Section ~\ref{errorsection} an error that measures how similar or different a subset's distance distribution is from that of the lattice itself.

For the upper bounds, we find specific configurations  of $p$ points that maximize the error and then calculate the error of these configurations, demonstrated through the calculations in Examples ~\ref{ex:p=4} through ~\ref{ex:checkerboard}. For the lower bounds, we may find a concrete bound in the case of small subsets of the lattice, giving us the bound in Equation ~\eqref{errorboundsmallsubset}. In the case of larger subsets, we take a more theoretical approach and construct theoretical optimal distance distributions, ones that cannot necessarily be realized by an actual subset of the lattice. We then describe the error given by these optimal distance distributions and prove a lower bound on error for some values of $p$. 

Thus, in this paper we seek to highlight preliminary results on this new perspective on the Erd\H{o}s distinct distances problem. We begin with the following section which introduces  some definitions and proves upper bounds for the frequency of the most common distance on the lattice.

\section{Introducing Distance Distributions}
\begin{figure}
    \centering
    \includegraphics[scale=.4]{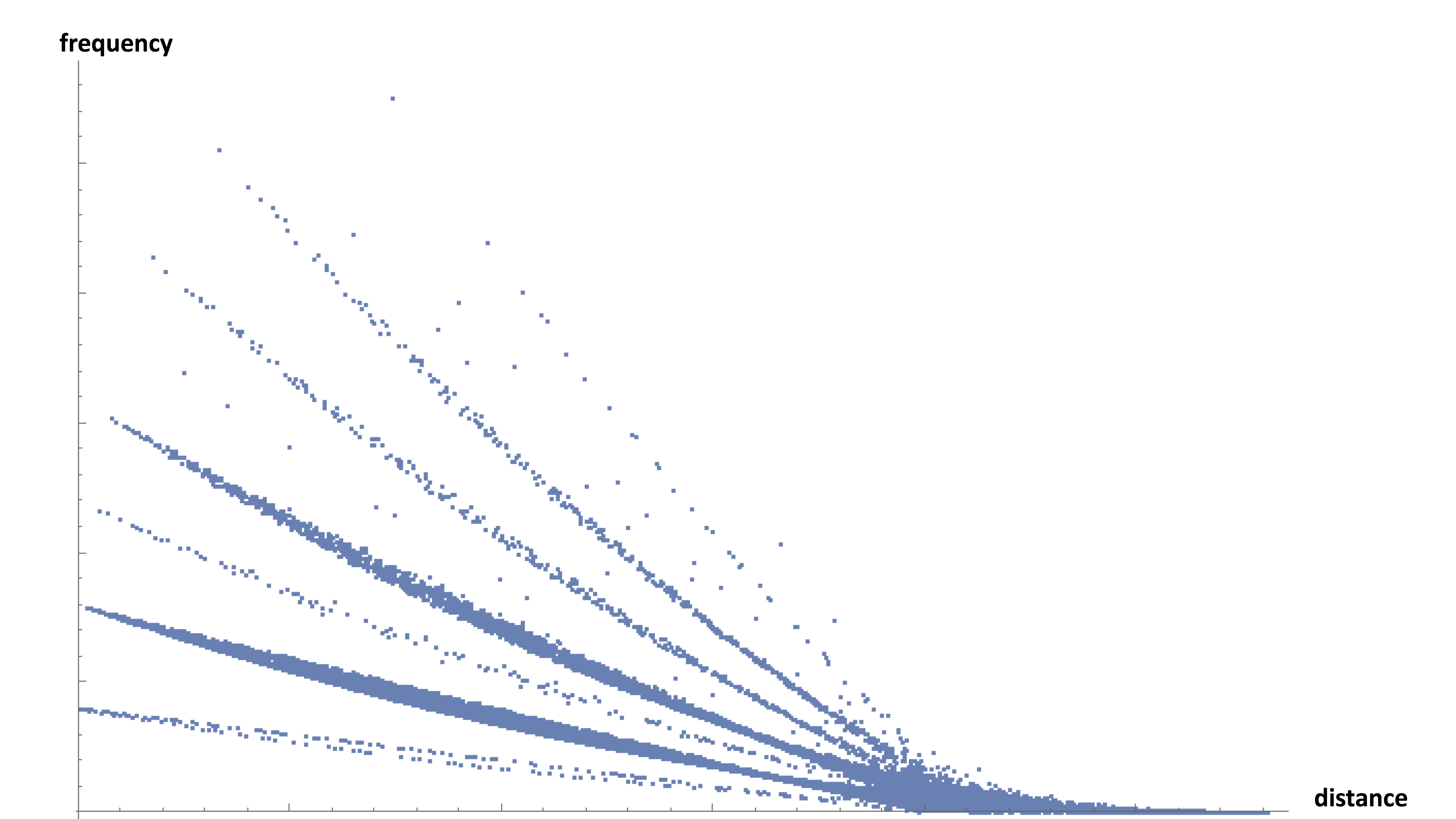}
    \caption{The distance distribution for the $200 \times 200$ lattice.}
    \label{fig:distance_distribution}
\end{figure}
Throughout this section, we characterize the distance distribution of the $N\times N$ integer lattice, as seen in Figure \ref{fig:distance_distribution}. For our characterization, we begin with some notational definitions and a lemma. 

\begin{definition}
	For a fixed $N$,  denote by $\mathcal{L}_N\subset \mathbb{Z}^2$ the $N\times N$ integer lattice, where \begin{equation}
		\mathcal{L}_N=\{ (x,y)\in \mathbb{Z}^2\ \mid  \ 0\leq x \leq N-1,\; 0\leq y\leq N-1 \}
	.\end{equation}
\end{definition}

\begin{definition}\label{def:D_n}
For a fixed $N$, denote by $\mathcal{D}_N$ the set of distinct distances which appear at least once on the $N\times N$ lattice.
\end{definition}
\begin{definition}
For any $d \ \geq \ 1$, denote by $L_{\sqrt{d}}$ the number of times that the distance $\sqrt{d}$ appears on $\mathcal{L}_N$.
\end{definition}

\begin{definition}
For any $d\geq 1$, denote by $S_{\sqrt{d}}$ the number of times that the distance $\sqrt{d}$ appears in a given subset $S\subseteq \mathcal{L}_N$.
\end{definition}

\begin{lemma}\label{distancecount}
For any $d\geq 1$, 
\begin{equation}\label{equationforLd}
	L_{\sqrt{d}} \ = \ \sum_{\substack{a^2+b^2=d \\ a\geq 1,\; b\geq 0}}^{} 2\left( N-a \right) \left( N-b \right).
\end{equation}

\end{lemma}

\begin{proof}
First, note that each occurrence of the distance $\sqrt{d}$ comes from two points $(w,x)$, $(y,z)$ satisfying $\sqrt{(w-y)^2 + (x-z)^2}=\sqrt{d}$. 

Let $a,b$ be an ordered pair with $a^2+b^2 = d$ and suppose both $a,b>0$. Notice that for any point $(w,x)$ on the integer lattice, $\left( w+a,x+b \right)  $ is in $\mathcal{L}_N$ if and only if $0\leq w\leq N-a-1$ and $0\leq x\leq N-b-1$; hence there are $(N-a)(N-b)$ such points $(w,x)$. The same logic can be repeated for $(w+a,x-b)$ to conclude that there are precisely $2(N-a)(N-b)$ pairs of points separated by an $x $-distance of $a$ and a $y$-distance of $b$.

In the case where $a>b=0$, by our assumption  that $a>0$, we need only find the number of pairs $(x,y) \in \mathcal{L}_N$ for which $(x+a,y)\in  \mathcal{L}_N,\ (x,y+a)\in \mathcal{L}_N$. There are $2N(N-a)$ such pairs. 

As this accounts for any possible occurrence of $\sqrt{d}$ on $\mathcal{L}_{N}$, we are done. As a final check, we note that the well-known identity

\begin{equation}
		\sum_{a=1}^{N-1}\sum_{b=0}^{N-1}2(N-a)(N-b) \ = \ \frac{N^2\left( N^2-1 \right) }{2} \ = \ \binom{N^2}{2}
	\end{equation}
confirms that we have counted all $\binom{N^2}{2}$ distances on the lattice.	
\end{proof}

As seen in Figure \ref{fig:distance_distribution}, the frequencies of distances are arranged in distinct curves. Clearly, the curve which $L_{\sqrt{d}}$ falls on is closely tied to the distinct ways $d$ can be written as the sum of two squares; if $d$ has $m$ representations as the sum of two squares, then $L_{\sqrt{d}}$ falls on the $m$th highest curve.
In fact, this is a subject which has been studied in some detail, which we summarize below.

\begin{definition}
	For any $n \in \mathbb{Z} $, let $r_2(d)$ be the number of ordered pairs $\left( a,b \right) \in \mathbb{Z}^2$ such that $a^2+b ^2=d$.
\end{definition}
We state the following classical result due to Fermat; for a survey of the many possible proofs of this theorem, see \cite{Co}.
\begin{theorem}[Fermat, 1640]\label{r2n}
	If $d$ is a positive integer with prime factorization $d=2^{f}p_1^{g_1}\cdots p_{m} ^{g_{m}}q_1 ^{h_1}\cdots q_{n} ^{h_{n}}$, where for any $i$, $p_{i}\equiv 1\pmod{4}$ and $q_{i}\equiv 3\pmod{4}$, then \begin{equation}
		r_2(d)=\begin{cases} 
			4\left( g_1+1 \right) \cdots \left( g_{m}+1 \right)  & \; {\rm all} \ {\rm the } \;  h_{i}\; {\rm are} \ {\rm even,} \;   \\
			0 & \; {\rm otherwise.} \;   \\
		\end{cases}
	\end{equation}
\end{theorem}

For some motivation towards this result, recall that the equation $a^2+b^2=p$, where $a,b>0$ and $p$ is prime, has a unique solution in the case where $p\equiv 1 \pmod{4}$ and no solution otherwise. The proof of the theorem relies on building inductively upon this result. The coefficient $4$ simply accounts for the options of $\pm a, \pm b$. 

\begin{rek}
	The pairs $(a,b)\in \mathbb{Z}^2$ which are counted in the above sum may be negative; this contradicts our original condition for ordered pairs $(a,b)$ in Lemma \ref{distancecount},  which only required $a\geq 1,\; b\geq 0$. This is intentional, as both quantities are calculated through different methods and are used for different purposes. In particular, because of the way these pairs are counted, we see that for a distance $\sqrt{d}$ on the $m$-th curve, $r_2(d)=4m$.
\end{rek}
\begin{definition}
For any $k\geq 1$, let $\sqrt{ n_{k} }$ be the smallest distance on the $k$-th curve, i.e., $n_{k}$ is the smallest positive integer such that $r_2(n_{k})=4k$.
\end{definition}
We use these resuls to find the most common distance on the integer lattice, which proves to be useful later. 

Recalling the original lattice distance distribution seen in Figure \ref{fig:distance_distribution}, we note that the most common distance on each individual curve is generally the leftmost, i.e., the smallest. For our estimates, then, we hope to take the smallest distance on the $k$-th curve---defined above as $n_k$---as an approximation of the most common distance on each curve. This proves a reasonable assumption: indeed, suppose two integers $d_1<d_2$ are such that $r_2(d_1)=r_2(d_2)=k$, yet $\sqrt{d_2}$ is the more common distance on some $N\times N$ integer lattice. Using Lemma \ref{distancecount}, this implies we may find at least two ordered pairs of integers $(a_1,b_1),(a_2,b_2)$ satisfying $a_i^2+b_i^2=d_i$ for each $i$, for which $2(N-a_1)(N-b_1) <2(N-a_2)(N-b_2) $, i.e., $a_2b_2-a_2-b_2>a_1b_1-a_1-b_1$. As $d_i \geq 2a_ib_i$ by the inequality of arithmetic and geometric means, with $d_i\leq 2\sqrt{2}a_ib_i$, we thus have a generous bound of $d_2 \leq \sqrt{2} d_1$. 

In our search for the most common distance on the lattice, we thus narrow our focus to the set of integers $n_1,n_2,\ldots$, and their asymptotic frequencies on the $N\times N$ lattice. By our previous argument, these $n_k$ values are at least within a constant of $\sqrt{2}$ of the true most common distance on the $k$-th curve. With these guiding principles, we now attempt to bound the size of the integers in this sequence.

\begin{lemma}
Let $p_1<p_2<\cdots$ be all the primes satisfying $p_{i}\equiv 1\pmod{4}$, listed in increasing order (so that $p_1=5,\ p_2=13,\ p_3=17$, and so on). Suppose $k=q_1^{a_1}\cdots q_{m}^{a_{m}}$, where $q_1,\ldots , q_{m}$ are any $m$ distinct primes and $q_1>q_2>\cdots >q_m$. Then an upper bound for $n_k$ is given by \begin{equation}
	\left( \underbrace{p_1\cdots p_{a_1}}_{a_1\; {\rm primes} \;  } \right) ^{q_1-1}\!\left( \underbrace{p_{a_1+1}\cdots p_{a_1+a_2}}_{a_2\; {\rm primes} \;  } \right) ^{q_2-1}\!\!\!\! \!\! \!\! \cdots \left( \underbrace{ p_{a_1+\cdots+a_{m-1}+1}\cdots p_{a_1+\cdots+p_{a_1+\cdots+a_{m}}}}_{a_{m}\; {\rm primes} \;  } \right) ^{q_{m}-1}.
	\end{equation}
\end{lemma}
\begin{proof}
	Let $n_k'$ be the bound presented in the lemma. As $n_k$ is the minimum value for which $r_2(n_k)=4k$, it suffices to show that $n_k'$ has this property, in order to determine that it is an upper bound. By a direct application of Theorem \ref{r2n},
	\begin{equation}
	r_2(n_k') \ = \ 4((q_1-1)+1)^{a_1}\cdots((q_m-1)+1)^{a_m} \ = \ 4k.
	\end{equation}\end{proof}

Note that the sequence $n_1', n_2',\ldots$ of bounds as presented by this lemma is not strictly increasing. In particular, the bounds in Lemma \ref{enserio} are attainable for infinitely many values of $k$. In particular, for $k=2^{m}$, we see $n_{k}'=p_1\cdots p_{m}=n_k$, as $r_2(2^{f}p_1^{g_1}\cdots p_{m} ^{g_{m}}q_1 ^{h_1}\cdots q_{n} ^{h_{n}})=4(g_1+1)\cdots (g_m+1)=4\cdot 2^m$ implies $g_1=\cdots=g_m=1$; additionally, for $k$ prime, $r_2(n)=r_2(2^{f}p_1^{g_1}\cdots p_{m} ^{g_{m}}q_1 ^{h_1}\cdots q_{n} ^{h_{n}})=4(g_1+1)\cdots (g_m+1)=4\cdot k$ implies $n$ has a sole prime divisor with exponent $k-1$.

It can be calculated that the proposed $n_k'$ bound matches the actual value for $n_k$ for most small values of $k$; in fact, we conjecture that the two are equivalent for the vast majority of $k$, although the closest result we may achieve is to show that the two sequences are asymptotic. The following lemma assists in finding explicit bounds for $n_k$. 

\begin{lemma}\label{enserio}
Let $p_1<p_2<\cdots$ be all the primes satisfying $p_{i}\equiv 1\pmod{4}$, listed in increasing order %(so that $p_1=5,\ p_2=13,\ p_3=17$, and so on)
. For any $k\geq 1$,
\begin{equation}\label{boundingnk}
	\prod_{i=1}^{\left\lfloor\log_2(k) \right\rfloor }p_{i} \ \ll \ n_{k} \ \leq \ 5^{k-1}
.\end{equation}
\end{lemma}
\begin{proof}
Firstly, as $r_2(5^{k-1})= 4(k-1+1)=4k$, we see immediately that $n_k \leq 5^{k-1}$. We proceed to show the lower bound.

Write $k=q_1^{a_1}\cdots q_{m}^{a_{m}}$, where $q_1,\ldots , q_{m}$ are $m$ distinct primes. Note that 
\begin{equation}
\left\lfloor\log_2(k) \right\rfloor  \ \leq \  a_1 + \cdots + a_m,
\end{equation} with equality if and only if $k=2^l$ for some $l \in \mathbb{N}$. 

Suppose $d$ is such that $r_2(d)=4k$. We wish to find the  minimum possible value for $d$.

We may assume that $d$ is not divisible by any primes that are congruent to 2 or 3 mod 4, as this would strictly increase its size without affecting the value of $r_2(d)$, as per Theorem \ref{r2n}. Hence, we may write $d=s_1 ^{g_1}\cdots s_t ^{g_t}$, for primes $s_i\equiv 1\pmod {4}$. Assuming without loss of generality that $g_1 \geq \cdots\geq g_t$, we see that the value of $d$ can be strictly decreased without affecting $r_2\left( d \right) =4\left( g_1+1 \right) \cdots \left( g_{t}+1 \right) $ by replacing $s_1=p_1,\ s_2=p_2,\ldots s_{t}=p_{t}$, as this will assign the highest exponents to the lowest possible prime values which are $1 \pmod {4}$. Hence we see $d$ is of the form 
\begin{equation}
d \ = \ p_1^{g_1}\cdots p_t ^{g_t} \ \approx \ p_t^{tg_t} \ \approx \ (t\log t)^{tg_t}.
\end{equation}
As $(g_1+1)\cdots (g_t+1)=k$, where $g_1\leq \cdots \leq g_t$, we see that $g_t^t \gg k$. Hence we may take $t\approx \log_{g_t}(k)$, and see that the quantity is maximized with large $t$ and small $g_t$. As the minimal possible value of $g_t$ is 1, which gives us $t\approx \log_2(k)$, we may take the asymptotic lower bound 
\begin{equation}\prod_{i=1}^{\log_2(k)} p_i.\end{equation}
\end{proof}
%explicar aqui

We wish to find a more explicit lower bound for $n_{k}$, which by Lemma \ref{enserio} is equivalent to estimating the product of the first $t$ primes $p_1<\cdots <p_{t}$ which are  congruent to $1\pmod{4}$. While this quantity is not well-studied, we may adapt existing work on bounding the product of the first $t$ primes $q_1,\ldots , q_{t}$ of any class $\pmod{4}$, denoted $q_{t}\#$.

The best general estimate is
\begin{equation}
	q_{t}\#=\prod_{i=1}^{t}q_{i}=e^{\left( 1+{o}(1) \right) t \log t}	
.\end{equation}

We can then approximate our quantity as 
\begin{equation}\label{estimate}
	\prod_{i=1}^{t}p_{i} \ = \ \left( \prod_{i=1}^{2t}q_{i} \right) ^{1/2} \ = \  e^{\frac{1}{2}\left( 1+o(1) \right)2t\log 2t } \ = \ e^{\left( 1+o(1) \right)t\log 2t }
.\end{equation}
Given the equal spread of the primes in any arithmetic progression, we know (\ref{estimate}) must converge towards the real quantity. 

Using this quantity, we have a general lower bound for $n_{k}$ using $t=\log_2(2k)$, namely 
\begin{equation}
	n_{k} \ \gg \ e^{\frac{1}{2}\left( 1+o(1) \right) \log_2(2k)\log \log_2(2k)} 
	%=  \log_2 (2k)^{\frac{1}{2} (1+o(1))\log_2(2k) } 
	\ =  \ (2k)^{\frac{1}{2}(1+o(1))\log_2(\log_2(2k))}
.\end{equation}
Now, given that a distance $\sqrt{d}$ is on the $k$-th curve of the distance distribution, we can maximize each summand in Lemma \ref{distancecount} to find the upper bound
\begin{equation} \label{eq:upperbound}
	L_{\sqrt{d}} \ \leq \ 2k N \left( N-\sqrt{d} \right) 
.\end{equation}
As (\ref{eq:upperbound}) assumes that all $k$ pairs of integers $(a,b)$ with $a^2+b ^2=d$ are identically $(\sqrt{d},0)$, we know that for $k>1$, (\ref{eq:upperbound}) is indeed a strict upper bound for this quantity.

Putting these pieces together, we can determine
\begin{equation}\label{upperbound}
	L_{\sqrt{n_{k}}} \ \ll \ 2k N\left( N-   (2k)^{\frac{1}{4}(1+o(1))\log_2(\log_2(2k))} \right) 
.\end{equation}
In particular, for large $N$, maximizing this quantity in terms of $k$  gives us a strict upper bound for the most common distance on the $N\times N$ lattice, which we denote $F_{N}$. Thus, we reach an asymptotic bound for the most common distance in the lattice, which proves useful in the following section.

\section{Error Estimates}\label{errorsection}
After studying the distance distribution for the lattice, the logical next step is to study behavior of distance distributions for subsets of the lattice. Although we know that the lattice is a near-optimal set, as previously discussed, the behavior of the distance distributions of its subsets can vary widely.

Thus, we examine how different and similar the distance distributions for subsets of the lattice can be to that of the lattice, an analogous version of the Erd\H{o}s distinct distances problem on subsets of the lattice.   

We now define our method of calculating the difference between the integer lattice's distance distribution and that of one of its subsets. Recall that  $L_{\sqrt{d}}$ is the frequency of a distance $\sqrt{d}$ in the lattice and $S_{\sqrt{d}}$ is its frequency in a particular subset $S$. We note that the $N \times N$ lattice has $N^2(N^2-1)/2 \approx N^4/2$ total distances and a subset with $p$ points has $p(p-1)/2 \approx p^2/2$ total distances. Thus we scale up each $S_{\sqrt{d}}$ by $N^4/p^2$ to have a distance distribution with about the same total number of frequencies as that of the lattice. Then, we sum $\left|(N^4/p^2)S_{\sqrt{d}}-L_{\sqrt{d}}\right|$ over all $d \in \mathcal{D}_N$.

More explicitly, we have the formula in the following definition. 

\begin{definition}
Let $\varepsilon$ be the error between the distance distribution of the integer lattice and one of its subsets. Then, \begin{equation}
\varepsilon \ = \ \frac{1}{|\mathcal{D}_N |} \sum_{d\in\mathcal{D}_N}\left|\frac{N^4}{p^2}S_{\sqrt{d}}-L_{\sqrt{d}}\right|
. \end{equation}
\end{definition}

In many of our calculations, instead of working with the actual distances, we work with individual pairs $(a,b)$ for which $\sqrt{a^2+b^2} \in \mathcal{D}_N$. This gives exact values of the contribution to the error for distances on the first and second curves; in all other cases, this is a simplification. Thus we introduce some new notation.  
\begin{definition}
Let $L_{a,b}$ denote the number of unordered pairs $(a_1, b_1)$, $(a_2, b_2) \in \mathcal{L}_N$ such that $|a_1-a_2|=a$ and $|b_1-b_2|=b$ or $|a_1-a_2|=b$ and $|b_1-b_2|=a$.

\end{definition}

\begin{definition}
Let $S_{a,b}$ denote the number of unordered pairs $(a_1, b_1)$, $(a_2, b_2)$ in a given subset $S\subset \mathcal{L}_N$ such that $|a_1-a_2|=a$ and $|b_1-b_2|=b$ or $|a_1-a_2|=b$ and $|b_1-b_2|=a$.
\end{definition}
\begin{definition}
Let $\varepsilon_{a,b}=|(N^4/p^2)S_{a,b}-L_{a,b}|$.
\end{definition}
 It may be shown that counting repeat distances as distinct either strictly increases error or has no effect on it. If $\sqrt{a^2+b^2}=\sqrt{c^2+d^2}$ for $\{a,b\}\neq\{c,d\}$, we then see that the total contribution to error for this distance is \begin{equation}\left|\frac{N^2}{p^2}S_{a,b} - L_{a,b} + \frac{N^2}{p^2}S_{c,d} - L_{c,d}\right|, \end{equation} whereas counting them as distinct gives a total contribution to error \begin{equation}\left|\frac{N^2}{p^2}S_{a,b} - L_{a,b}\right| + \left|\frac{N^2}{p^2}S_{c,d} - L_{c,d}\right|.\end{equation} Counting each pair as distinct thus gives an upper bound for total error. Furthermore, these two quantities are identical if and only if $\left(N^2/p^2\right)S_{a,b} - L_{a,b}$ and $\left(N^2/p^2\right)S_{c,d} - L_{c,d}$ have the same sign, which we suspect is true of most subsets which maximize error.

Recall from the previous calculations we made on the frequency of distances on the lattice that when $b=0$ or $a=b$,   $L_{a,b}=2(N-a)(N-b)$. Otherwise, for $b>0$ and $a>b$, we have $L_{a,b}=4(N-a)(N-b)$.
For later error calculations, we need the average values of $2(N-a)(N-b)$  and $4(N-a)(N-b)$ and the fraction of $L_{a,b}$ that are of the form $2(N-a)(N-b)$ and the fraction of  $L_{a,b}$ that are of the form $4(N-a)(N-b)$. We provide these values in the following two lemmas, which follow immediately by induction. 

\begin{lemma} \label{lem:avg_val}
The average value of $2(N-a)(N-b)$ over $a=b$ or $b=0$ is \begin{equation}\left(2N\right)^{-1}\left(\sum_{a=1}^N 2(N-a)^2+ \sum_{a=1}^N 2N(N-a)\right)=\frac{N(5N-1)}{6}\end{equation}  and the average value of $4(N-a)(N-b)$ over $b>0$ and $a>b$ is \begin{equation}\left(\frac{1}{2}N(N-1)\right)^{-1}\left(\sum_{b=1}^{N-1} \sum_{a=b+1}^N 4(N-a)(N-b)\right)= \frac{N(3N-1)}{3}.\end{equation}
\end{lemma}

\begin{lemma}\label{lem:frac}
 The fraction of $L_{a,b}$ that are of the form $2(N-a)(N-b)$ is $4 / (N+2)$ and the fraction of  $L_{a,b}$ that are of the form $4(N-a)(N-b)$ is $(N-2) / (N+2)$.
\end{lemma}

Similarly to the original Erd\H{o}s  distance problem, we are interested in finding upper and lower bounds for the behavior we are studying. For the upper bounds, we find patterns of configurations that maximize the error. We  then calculate the error for these specific configurations. For the lower bounds, we construct a theoretical optimal distance distributions and calculate a lower bound on its error.

\section{Bounds}
We begin this section with explicit calculations of upper bounds on error for configurations of $p$ points.
A configuration of $p$ points that maximizes error needs to have as different a distance distribution as possible from the original full lattice when scaled up. Thus, the ratio of each distance's frequency to the total number of distances, including repeated distances, must be as different as possible from that of the full lattice. Specifically, we want to have a subset that has many distances that were infrequent in the full lattice and minimal number of distances that were very frequent in the full lattice.

For certain values of $p,$ by brute-force calculations we know the configuration that maximizes error. See Figures \ref{fig:p=4}, \ref{fig:p=5}, \ref{fig:p=9}, \ref{fig:p=4(N-1)}, \ref{fig:p=4(N-1)+4(N-3)}, and \ref{fig:p=ceiling[N.2]}.

\begin{figure}[!ht]
     \centering
     \begin{floatrow}
       \ffigbox[\FBwidth]{\caption{$p=4$.}\label{fig:p=4}}{%
         \includegraphics[scale=.2]{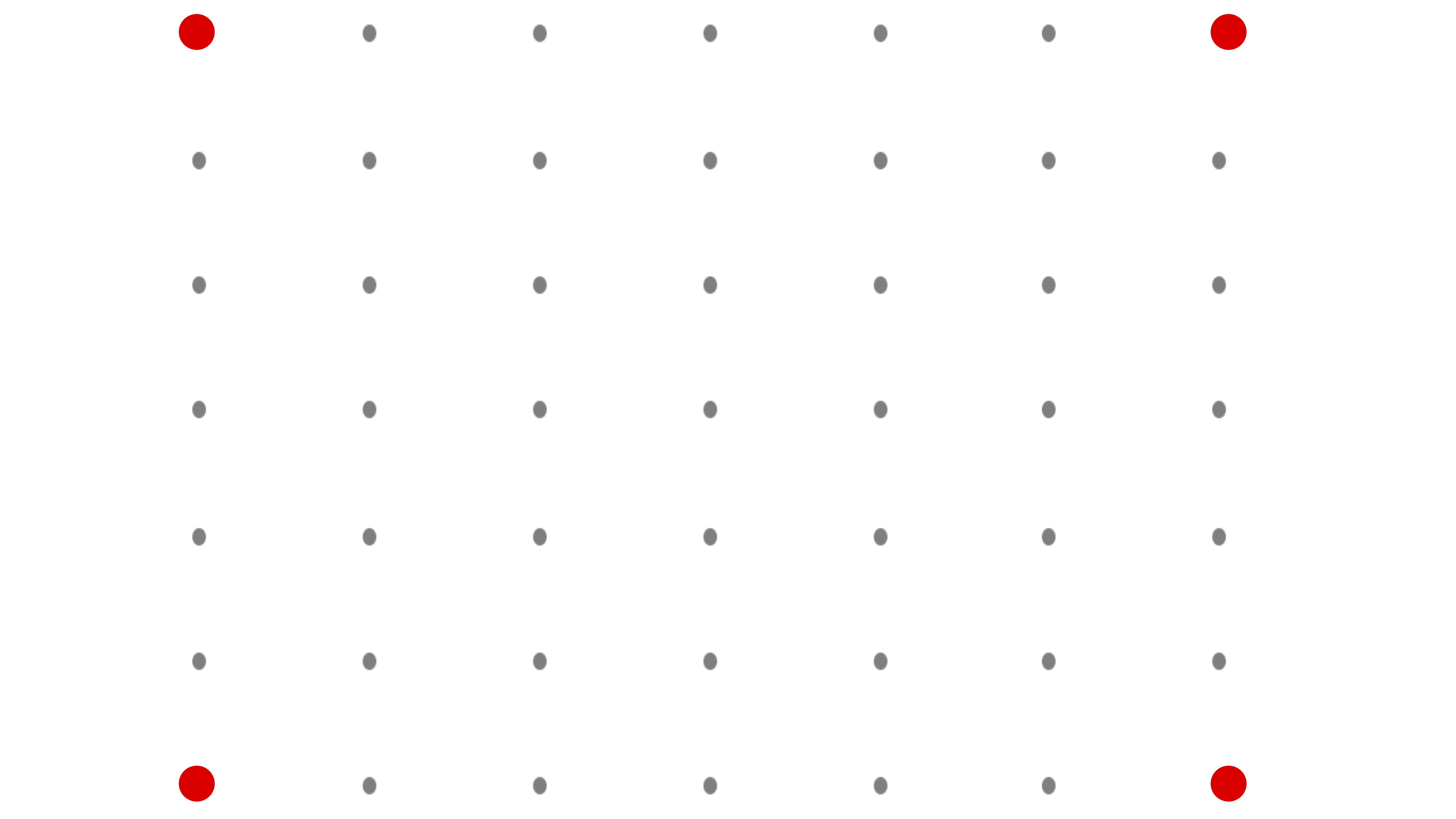}   % Just a dummy. Replace with your figure.
       }
       \ffigbox[\FBwidth]{\caption{$p=5$.}\label{fig:p=5}}{%
         \includegraphics[scale=.2]{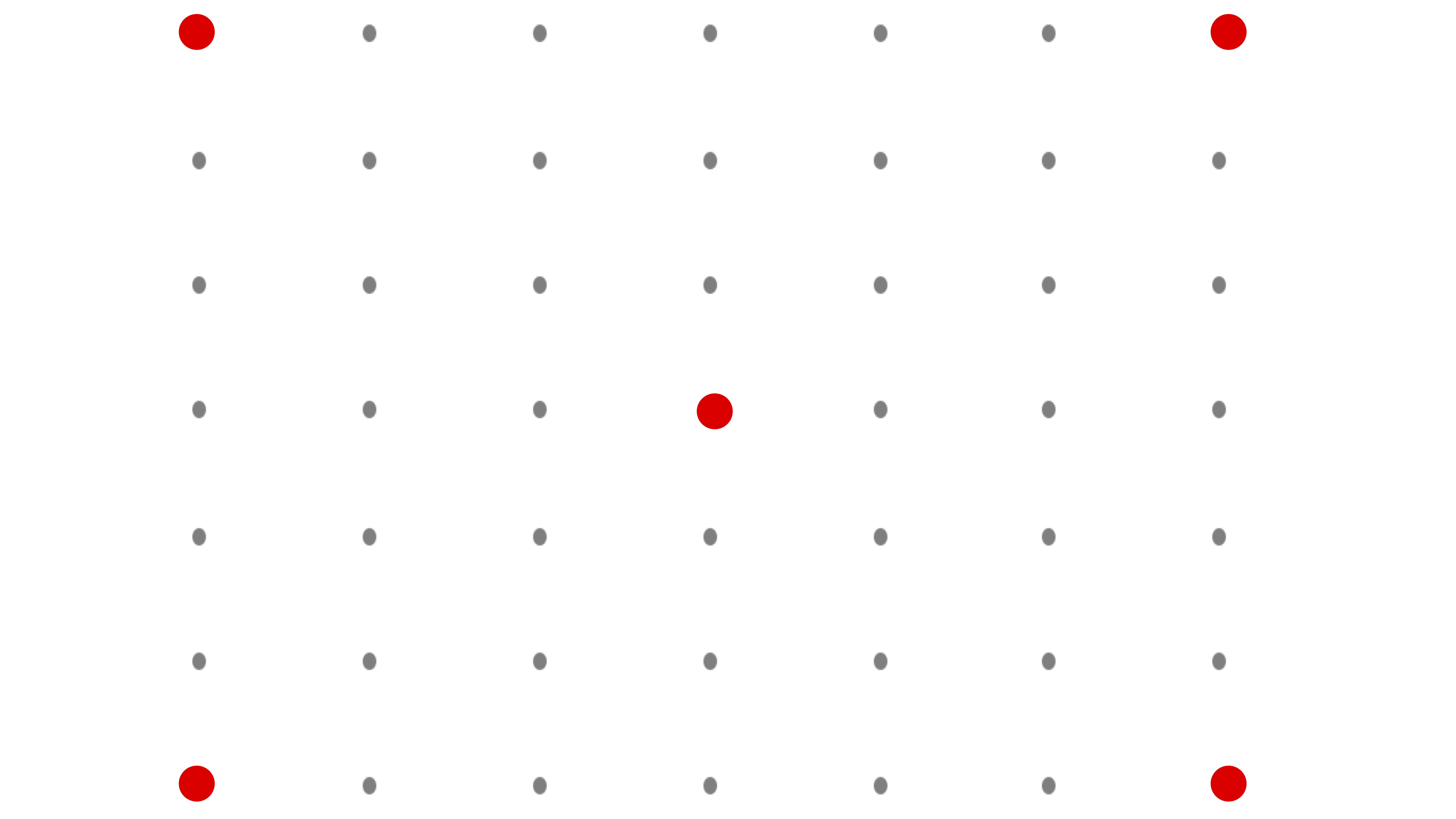}   % Just a dummy. Replace with your figure.
       }
     \end{floatrow}
  
     \centering
     \begin{floatrow}
       \ffigbox[\FBwidth]{\caption{$p=9$.}\label{fig:p=9}}{%
         \includegraphics[scale=.2]{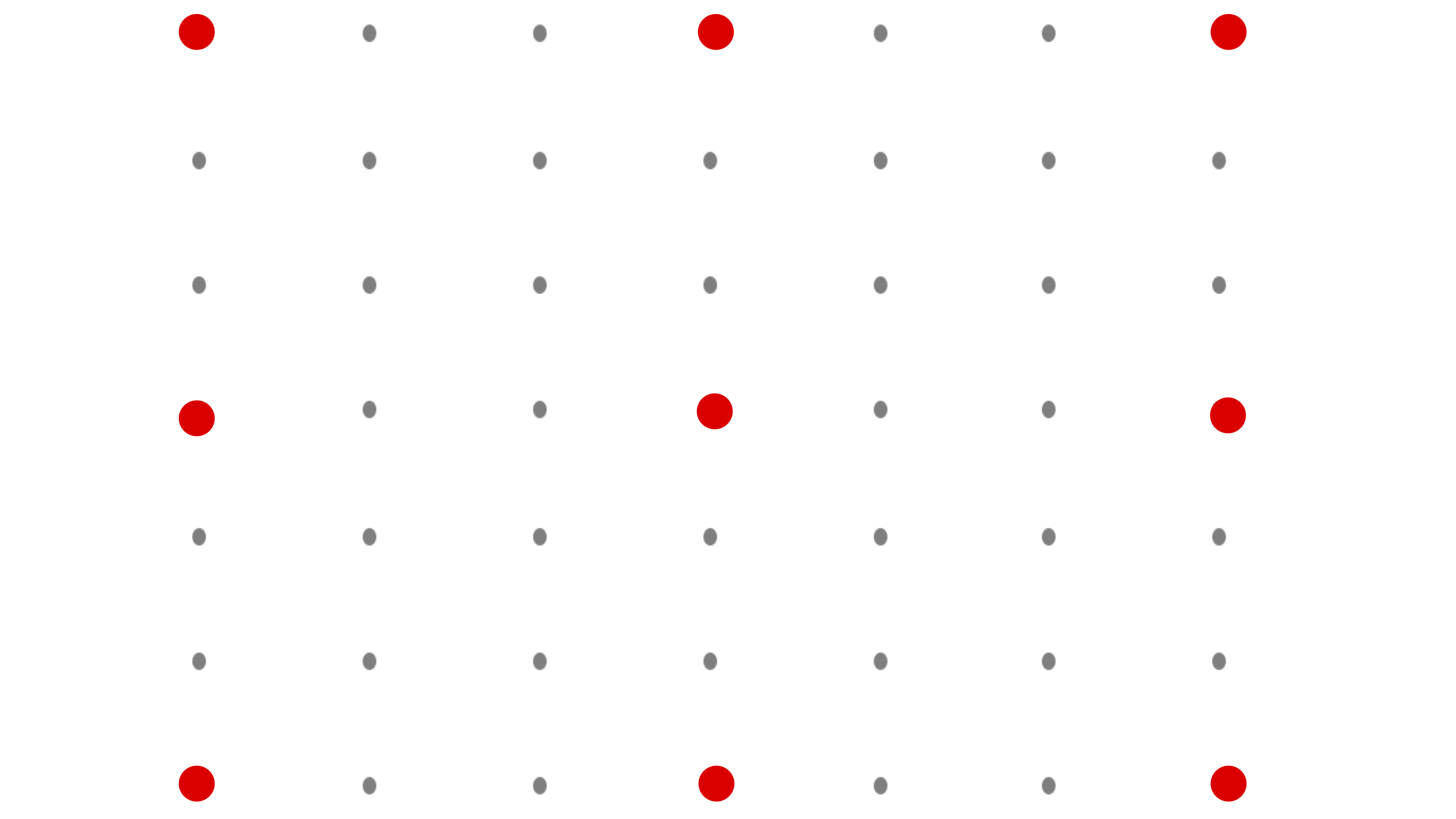}   % Just a dummy. Replace with your figure.
       }
       \ffigbox[\FBwidth]{\caption{$p=4(N-1)$.}\label{fig:p=4(N-1)}}{%
         \includegraphics[scale=.2]{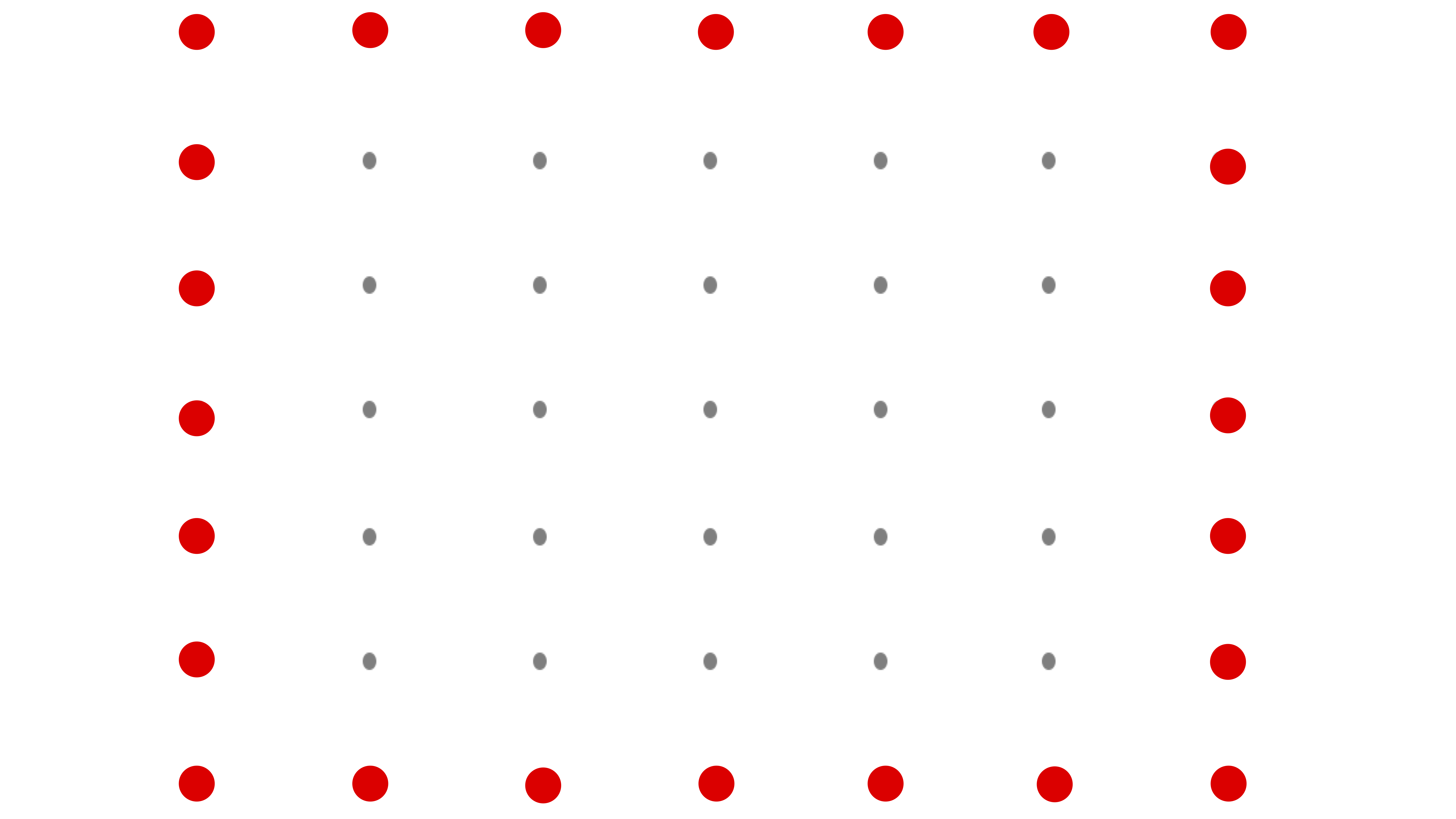}   % Just a dummy. Replace with your figure.
       }
     \end{floatrow}

     \centering
     \begin{floatrow}
       \ffigbox[\FBwidth]{\caption{$p=4(N-1)+4(N-3)$.}\label{fig:p=4(N-1)+4(N-3)}}{%
         \includegraphics[scale=.2]{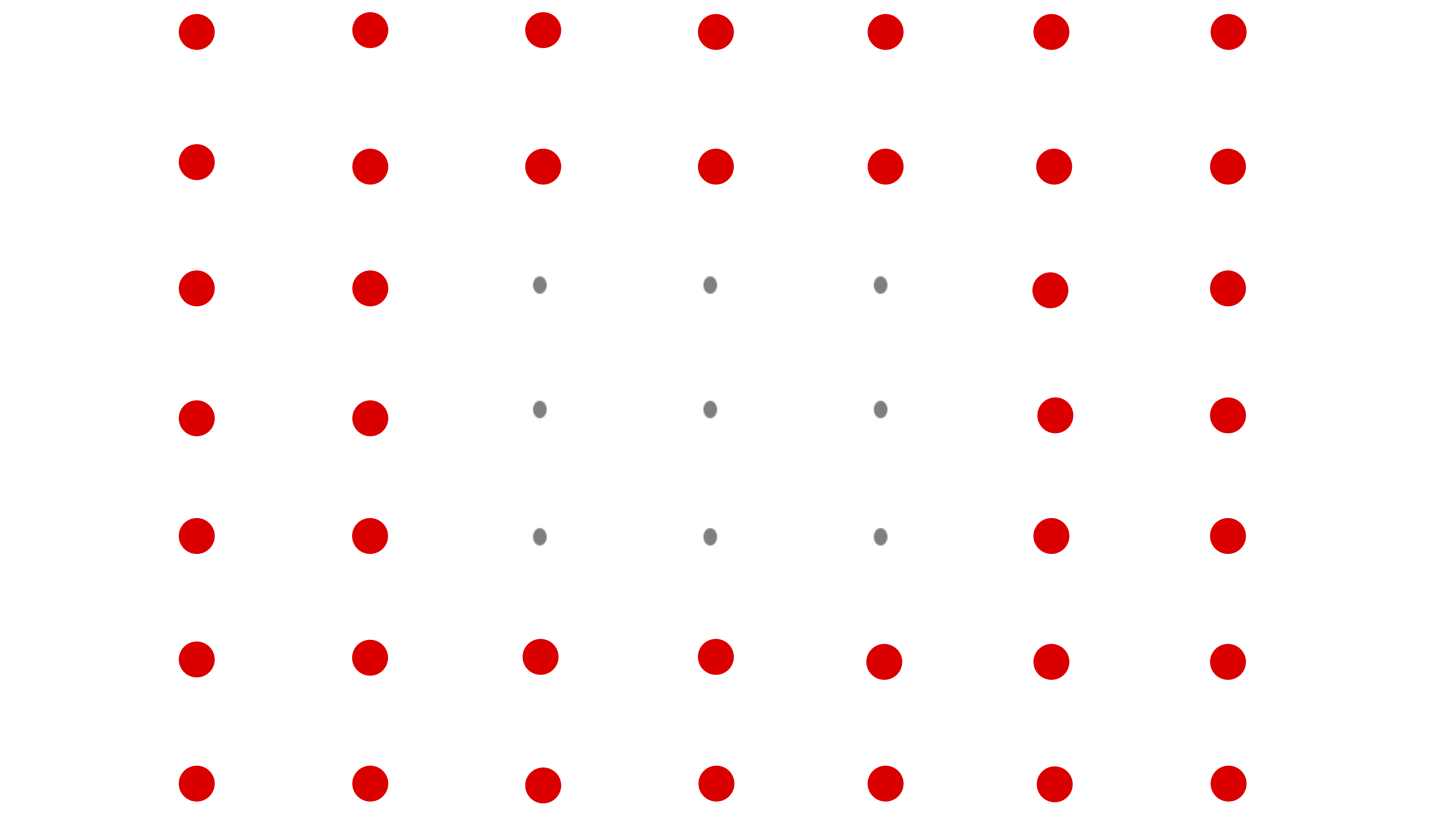}   
       }
       \ffigbox[\FBwidth]{\caption{$p=\left \lceil\frac{N^2}{2} \right \rceil$.}\label{fig:p=ceiling[N.2]}}{%
         \includegraphics[scale=.2]{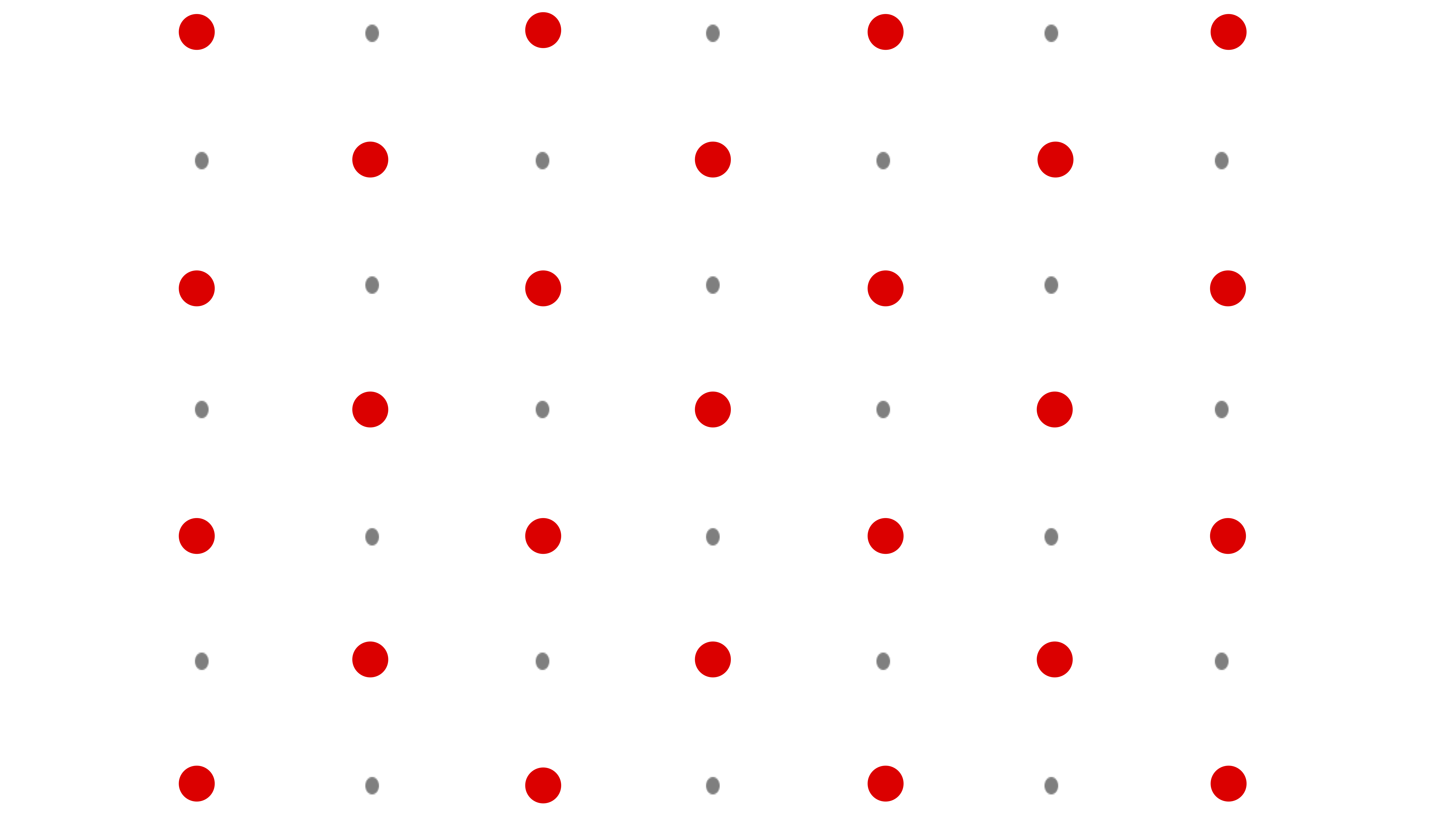}  
       }
     \end{floatrow}
   \end{figure}

%   \begin{figure}[!ht]
%      \centering
%      \begin{floatrow}
%       \ffigbox[\FBwidth]{\caption{$p=4$.}\label{fig:p=4}}{%
%          \includegraphics[scale=.2]{Error Estimates/p=4.png}   % Just a dummy. Replace with your figure.
%       }
%       \ffigbox[\FBwidth]{\caption{$p=5$.}\label{fig:p=5}}{%
%          \includegraphics[scale=.2]{p=5}   % Just a dummy. Replace with your figure.
%       }
%      \end{floatrow}
%   \end{figure}
  
%   \begin{figure}[!ht]
%      \centering
%      \begin{floatrow}
%       \ffigbox[\FBwidth]{\caption{$p=9$.}\label{fig:p=9}}{%
%          \includegraphics[scale=.2]{Error Estimates/p=9.png}   % Just a dummy. Replace with your figure.
%       }
%       \ffigbox[\FBwidth]{\caption{$p=4(N-1)$.}\label{fig:p=4(N-1)}}{%
%          \includegraphics[scale=.2]{Error Estimates/1fringe.png}   % Just a dummy. Replace with your figure.
%       }
%      \end{floatrow}
%   \end{figure}

%  \begin{figure}[!ht]
%      \centering
%      \begin{floatrow}
%       \ffigbox[\FBwidth]{\caption{$p=4(N-1)+4(N-3)$.}\label{fig:p=4(N-1)+4(N-3)}}{%
%          \includegraphics[scale=.2]{Error Estimates/2fringe.png}   % Just a dummy. Replace with your figure.
%       }
%       \ffigbox[\FBwidth]{\caption{$p=\left \lceil\frac{N^2}{2} \right \rceil$.}\label{fig:p=ceiling[N.2]}}{%
%          \includegraphics[scale=.2]{Error Estimates/checkerboard.png}   % Just a dummy. Replace with your figure.
%       }
%      \end{floatrow}
%   \end{figure}

We briefly describe these subsets. For $p=4,$ the maximal error subset is all four corners of the lattice. To transition to $p=5$, the middle point is added in. For $p=9$, the maximal subset is a $3 \times 3$ lattice stretched to the size of the $N \times N$ lattice. To transition from $p=5$ to $p=9$, the points that are in $p=9$, but not in $p=5$ are added in one by one. We then know for $p=4(N-1)$, the maximal error subset is the perimeter of the lattice, with no other points. For $p=\sum_{i=1}^m 4(N-(2i-1))$, the maximal error subset is the filled-in perimeter with a depth of $i$ points. For example, Figure \ref{fig:p=4(N-1)+4(N-3)} is a filled-in perimeter with a depth of 2 points. To transition from $p=\sum_{i=1}^m 4(N-(2i-1))$ to $p=\sum_{i=1}^{m+1} 4(N-(2i-1))$, the points only present in the latter configuration are filled-in one by one. The final maximal error configuration we found is for $p=\left \lceil N^2/2 \right \rceil$, for which every other point is filled-in to make a configuration we refer to as a checkerboard lattice. One can note that depending on the value of $N$, $\left \lceil N^2/2 \right \rceil$ is less than $p=\sum_{i=1}^{m} 4(N-(2i-1))$ for different values of $m$. As a result, there is a transition between these two types of configurations and a transition back.

In the following examples, we calculate error estimates for some of these configurations. For these calculations, we use the simplification of working with $\sqrt{a^2+b^2}$ where $0 \leq b \leq N-1$ and $b \leq a \leq N-1$, excluding  $a=b=0$, instead of distinct distances.  
\begin{example} \label{ex:p=4}
We begin with  $p=4$, where the maximal error configuration is a point in each  corner of the lattice. We first note that the scaling constant is   $N^4/p^2=N^4/16$.
We also have just two distinct distances, \begin{equation}\sqrt{(N-1)^2+(N-1)^2} \ = \ \sqrt{2}(N-1),\end{equation} which has $L_{N-1,N-1}=2$ and $S_{N-1,N-1}=2$, and \begin{equation}\sqrt{(N-1)^2+0^2} \ = \ N-1,\end{equation} which has $L_{N-1,0}=2N$ and $S_{N-1,0}=4$. To find $\varepsilon_{N-1,N-1}$ and $\varepsilon_{N-1,0}$, we have to scale  $S_{N-1,N-1}$ and $S_{N-1,0}$ by $N^4/16$ and subtract $L_{N-1,N-1}$ and $L_{N-1,0}$, respectively. This gives us $\varepsilon_{N-1, N-1}=N^4/8-2$ and $\varepsilon_{N-1, 0}=N^4/4-2N.$ Thus the total error contribution from these two distances is $3N^4/8-2-2N$.

Both $L_{N-1,N-1}$ and $L_{N-1,0}$ are of the form $2(N-a)(N-b)$, so we have to update the average value of $2(N-a)(N-b)$ from Lemma \ref{lem:avg_val} and the fraction of the time  $L_{a,b}$ is of the form $2(N-a)(N-b)$ from Lemma \ref{lem:frac} to exclude $L_{N-1,N-1}$ and $L_{N-1,0}$. The new average is $(5N^2+4N+3)/6$ and the new fraction is $(4N-8)/(N^2+N-2)$. The fraction of $L_{a,b}$ that are $L_{N-1,N-1}$ and $L_{N-1,0}$ is $4/(N^2+N-2)$. We note that for $(a,b)$ not equal to $(N-1,N-1)$ or $ (N-1,0)$, $S_{a,b}=0$. Thus, the average error contribution for these distances is their average frequency in the lattice.

We then put everything together to get our error estimate.
\begin{align}
\varepsilon & \ \leq \ \frac{4}{N^2+N-2}\left(\frac{3N^4}{8}-2-2N\right)+\frac{4N-8}{N^2+N-2}\left(\frac{5N^2+4N+3}{6}\right)\nonumber\\
&\ \ \ \ +\frac{N-2}{N+2}\left(\frac{N(3N-1)}{3}\right)\nonumber \\
& \ = \ \frac{5N^2}{2}-\frac{5N}{2}-\frac{15}{2(N-1)}-\frac{16}{N+2}+\frac{13}{2}.
\end{align}

This error estimate is an overestimate of the error when $N$ is small because of the fact that we are looking at $\sqrt{a^2+b^2}$, instead of distinct distances. However, the only distances affected by  this way of estimating the error are the two distances present on the lattice, $\sqrt{2}(N-1)$ and $N-1$ . As $N \rightarrow \infty$, the fraction of the total distances that these distances represent goes to zero. Thus, this error estimate converges to the actual error.
\end{example}

\begin{example} \label{ex:p=5}
Similarly, we can calculate the error when $p=5$. Recall, for this value of $p$, that the subset configuration that maximizes error is the four corners and the middle point of the lattice. 

This configuration has 3 distinct distances: \begin{equation}\sqrt{(N-1)^2+0^2} \ = \ N-1,\end{equation} for which we have  $S_{N-1, 0}=4$  and $L_{N-1, 0}=2N$, \begin{equation}\sqrt{(N-1)^2+(N-1)^2} \ = \ \sqrt{2}(N-1),\end{equation} for which we have   $S_{N-1, N-1}=2$  and   $L_{N-1, N-1}=2$, and \begin{equation}\sqrt{((N-1)/2)^2+((N-1)/2)^2} \ = \ \sqrt{2}(N-1)/2,\end{equation} for which we have  $S_{(N-1)/2, 0}=4$ and $L_{(N-1)/2, 0}=N+1$.

To calculate $\varepsilon_{N-1, 0},$ $\varepsilon_{N-1, N-1}$, and $\varepsilon_{(N-1)/2, 0}$, we need to multiply $S_{a,b}$ by $N^4/25$ and subtract $L_{a,b}$. Thus, $\varepsilon_{N-1, 0}=4N^4/25-2n$, $\varepsilon_{N-1, N-1}=2N^4/25-2$ and $\varepsilon_{(N-1)/2, 0}=4N^4/25-(N+1)$. 

We know that $L_{N-1, 0},$ $L_{N-1, N-1}$, and $L_{(N-1)/2, 0}$ are of the form $2(N-a)(N-b),$ so we need to edit the average value of $2(N-a)(N-b)$ from Lemma \ref{lem:avg_val}
and the fraction of the time   $L_{a,b}$ is of the form $2(N-a)(N-b)$ from Lemma \ref{lem:frac} to no longer include the three distances listed above.
The new average value is \begin{equation}\frac{5N^3-6N^2-8N-9}{3(2N-5)}\end{equation} and  the fraction of the time  $L_{a,b}$ is of the form $2(N-a)(N-b)$ is \begin{equation}\frac{6}{N+2}-\frac{2}{N-1}.\end{equation} The fraction of the total distances that are the three distances listed above  is \begin{equation}\frac{2}{N-1}-\frac{2}{N+2}.\end{equation}
We then put this all together to calculate the error:
\begin{align}
\varepsilon & \ \leq \ \left(\frac{2}{N-1}-\frac{2}{N+2}\right) \left[\left(\frac{4N^4}{25}-2N\right)+\left(\frac{2N^4}{25}-2\right)+\left(\frac{4N^4}{25}-(N+1)\right)\right]\nonumber\\
&\ \ \ +\left(\frac{6}{N+2}-\frac{2}{N-1}\right)\left[\frac{5N^3-6N^2-8N-9}{3(2N-5)}\right] +\frac{N-2}{N+2}\left[\frac{N(3N-1)}{3}\right]\nonumber \\
% =&\left(\frac{2}{N-1}-\frac{2}{N+2}\right) \left(\frac{2N^4}{5}-3N-3\right)+\left(\frac{6}{N+2}-\frac{2}{N-1}\right)\left(\frac{5N^3-6N^2-8N-9}{3(2N-5)}\right)\\
% +&\frac{N-2}{N+2}\left(\frac{N(3N-1)}{3}\right)\\
& \ = \ \frac{17N^2}{5}-\frac{17N}{5}-\frac{6}{N-2}-\frac{56}{5(N-1)}-\frac{124}{5(N+2)}-\frac{31}{3(2N-5)}+\frac{113}{15}.\end{align}

Similarly to $p=4$, this error estimate overestimates  the error when $N$ is small because of the fact that we are looking at $\sqrt{a^2+b^2}$, instead of distinct distances $\sqrt{d}$. However, as $N \rightarrow \infty$, this error estimate converges to the actual error.
\end{example}

\begin{example}\label{ex:p=9}
We can then examine what happens when $p=9$. The 9 point configuration that maximizes error is a $3 \times 3$ lattice that has been stretched to the size of the $N \times N$ lattice.

We first note that $S_{a,b}>0$  if and only if $a,b\equiv 0 \pmod{(N-1)/2}.$ Thus, the fraction of $S_{a,b}$ such that $S_{a,b} \neq 0$ for $b\neq 0$ and $a >b$  is \begin{equation}\frac{4}{(N-1)^2}\end{equation} and the fraction of $S_{a,b}$ such that $S_{a,b}\neq 0$ for $b=0$ or $a=b$ is \begin{equation}\frac{2}{N-1}.\end{equation} The scaling constant is $N^4/81$. We estimate the error from above by assuming that if $S_{a,b}\neq 0$, then $S_{a,b}=L_{a,b}.$ This assumption does not increase the error estimate by an unreasonable amount because, for large enough $N$, the fraction of total distances that are represented in this configuration is very low. We can then use our previous averages of $L_{a,b}$ to calculate error:
\begin{align}
\varepsilon & \ \leq \ \frac{4}{N+2}\Bigg[\frac{2}{N-1}\left(
\frac{N^4}{81}\left(\frac{N(5N-1)}{6}\right)-\frac{N(5N-1)}{6}\right)\nonumber \\
&\ \ \  \ +\left(1-\frac{2}{N-1}\right)\left(\frac{N(5N-1)}{6}\right)  \Bigg] \nonumber \\
&\ \ \ \  +\frac{N-2}{N+2}\Bigg[\frac{4}{(N-1)^2}\left(\frac{N^4}{81}\left(\frac{N(3N-1)}{3}\right)-\frac{N(3N-1)}{3}\right)\nonumber\\
&\ \ \ \  +\left(1-\frac{4}{(N-1)^2}\right)\left(\frac{N(3N-1)}{3}\right)  \Bigg] \nonumber \\
& \ = \ \frac{32N^4}{243}-\frac{52N^3}{243}+\frac{4N^2}{9}-\frac{220N}{243}-\frac{23044}{2187(N-1)}-\frac{14000}{2187(N+2)}\nonumber \\
&\ \ \ \  -\frac{6200}{729(N-1)^2}+\frac{112}{27(N-1)^3}+\frac{32}{9(N-1)^4}+\frac{428}{243}.\end{align}
\end{example}

\begin{example}\label{ex:checkerboard}
Finally, we examine the error for the configuration for $p=\left\lceil N^2/2\right \rceil$. This configuration is the \textquotedblleft checkerboard lattice," a subset that is missing every other point from the full lattice and resembles a checkerboard, as seen in Figure \ref{fig:p=ceiling[N.2]}. 

To provide some intuition, the checkerboard lattice is a reasonable configuration for maximizing error because it very strongly prioritizes distances where $b=0$ or $a=b$. These distances have $L_{a,b}=2(N-a)(N-b)$, which tend to be smaller than other frequencies where $L_{a,b}=4(N-a)(N-b)$. Thus, this configuration has many frequencies which are not very common in the full lattice.

In a checkerboard, $\sqrt{a^2+b^2}$ only appears as a distance if either  $a$ and $b$ are both odd or both even. We note that is this causes about $1/2$ of $S_{a,b}$ to be zero for $a> b$ and $b \neq 0$. Additionally, we note that this causes about $1/4$ of  $S_{a,b}$ to be zero for $a= b$ or $b = 0$. We use the simplifying assumption that $S_{a,b}=L_{a,b}$ if $S_{a,b}\neq 0$. This ultimately increases our error estimate. We then have the following error estimate.
\begin{align}
\varepsilon & \ \leq \   \frac{4}{N+2}\left[\frac{3}{4}\left(4\left(\frac{N(5N-1)}{6}\right)-\frac{N(5N-1)}{6} \right)+\frac{1}{4} \left(\frac{N(5N-1)}{6}\right)  \right]\nonumber\\
 &\ \ \ \ + \frac{N-2}{N+2}\left[ \frac{1}{2}\left(4\left(\frac{N(3N-1)}{3}\right) -\frac{N(3N-1)}{3}\right)+\frac{1}{2} \left(\frac{N(3N-1)}{3}\right)  \right]\nonumber\\
& \ = \ 2N^2-\frac{N}{3}-\frac{2}{3(N+2)}+\frac{1}{3}.\end{align}

In Appendix \ref{checkerboardcalculations}, we more precisely calculate the frequency of the distances on the checkerboard lattice. These calculations can be used for a closer estimate. 
\end{example}

We now turn our attention to finding lower bounds for configurations of $p$ points. To minimize error, we want to preserve the same ratio of total distances, including repeated distances, to frequency that appeared in the $N \times N$ lattice for each unique distance. Thus, to calculate  a lower bound, we create an \textquotedblleft optimal" distribution of frequencies for $p$ points.   This optimal distribution cannot always be achieved, as not every distance distribution is realizable by a certain configuration.  To come up with the optimal distribution, we scale each $L_{a,b}$ by $N^4/p^2$ and round this number to the nearest integer. We then find the error for this optimal distribution in the same way as before. Code can easily be written to find the optimal distance distribution and calculate the error. For example, Figure \ref{fig:lower_bound_theoretical} demonstrates what this error is for a $100 \times 100$ lattice. Note that this figure does not include all possible values of $p$, rather it includes enough values to capture the behavior of the error.

\begin{figure}
    \centering
    \includegraphics[scale=.35]{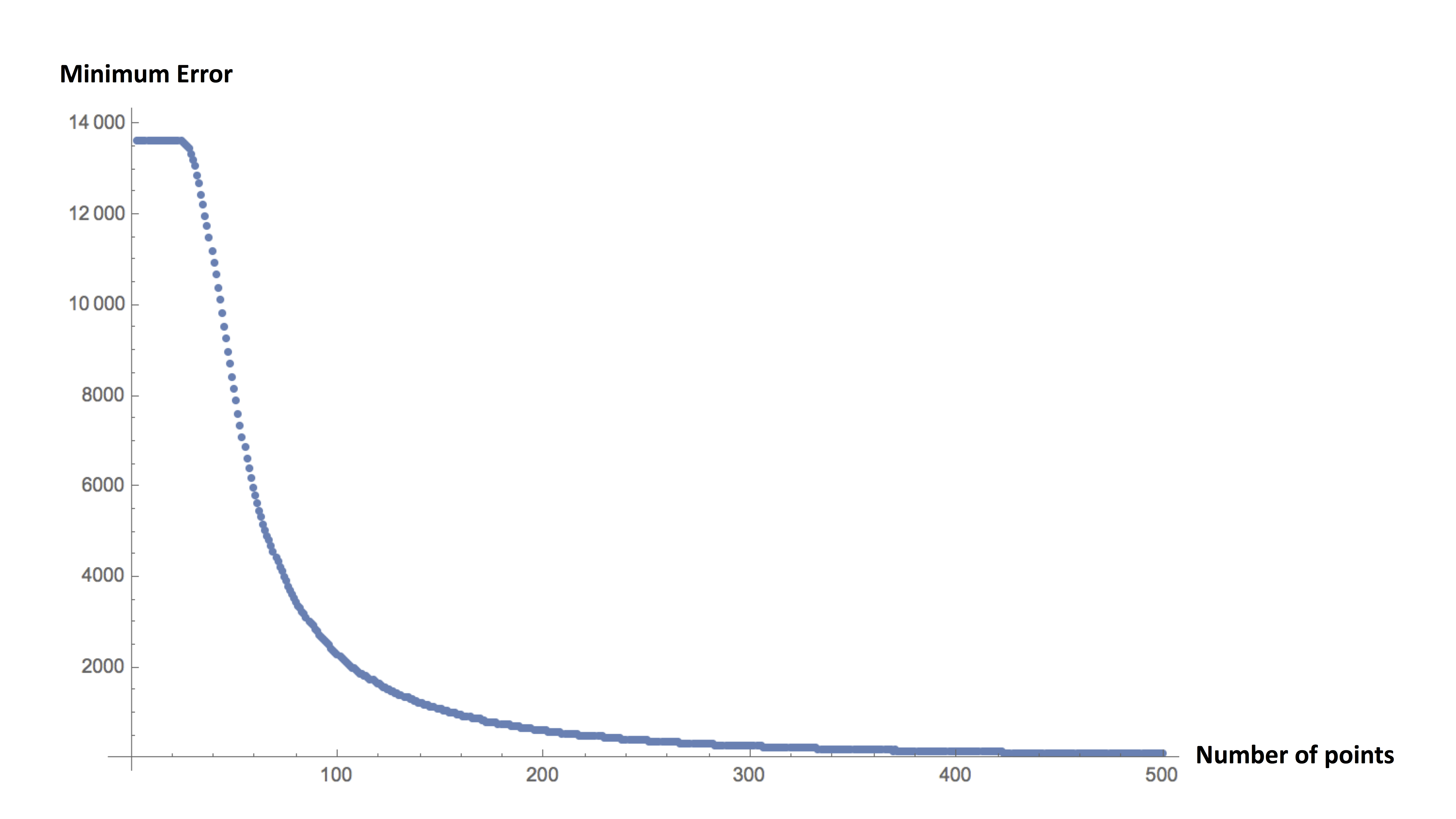}
    \caption{Computer generated data for the error of the optimal distance distribution in the $100 \times 100$ lattice.}
    \label{fig:lower_bound_theoretical}
\end{figure}

We begin with providing intuition for the behavior of the second part of the graph, the decreasing curve for large $p$.

Once the $L_{a,b}$'s are scaled down by $p^2/N^4$, rounded, and scaled up by $N^4/p^2$, they are a multiple of $N^4/p^2$. That means that the largest $\varepsilon_{a,b}$ can be is $N^4/2p^2$ and the smallest $\varepsilon_{a,b}$ can be is $0$. One might expect the average contribution to error to be $N^4/4p^2.$ However, for small $p$, many frequencies in the $N \times N$ lattice are much closer to $0$ than $N^4/2p^2$, as $N^4/2p^2$ is quite large. This explains the shape of the graph, however it does not provide a concrete value for the error.

On the other hand, for the smaller values of $p$, we can give a more precise description. Notice that the graph in Figure \ref{fig:lower_bound_theoretical} begins with a short horizontal line. More rigorously, we have \begin{equation}\label{errorboundsmallsubset}
\varepsilon \ \geq \  \binom{N^2}{2} \text{ if } p \ \leq \ \frac{N^2}{\sqrt{2F_N}}.
\end{equation}

Here, we note that $\binom{N^2}{2}$ is the precise value for error for the empty subset of the lattice, i.e.,  the sum of distance frequencies of the lattice itself. The idea is that the distance distribution of a small enough subset of the lattice, after rescaling, has a greater error than the empty subset, as its few distances are very overrepresented.

As we use a scaling factor of $N^{4}/p^2$, we notice that if $p$ is such that
\begin{equation}
\frac{N^{4}}{p^2} \ > \ 2L_{\sqrt{d}}
\end{equation}
for any distance $\sqrt{d}$ on the integer lattice, then the error of any subset of size $p$ is strictly greater than that of the empty subset, as for any $\sqrt{d}$, \begin{equation}
\left| \frac{N^{4}}{p^2}S_{\sqrt{d}}-L_{\sqrt{d}} \right| \ \geq \ \left| \frac{N^{4}}{p^2} - L_{\sqrt{n_{k}}} \right| \ \geq \ F_N
,\end{equation}
where, as previously defined, $F_N$ denotes the highest frequency on the $N\times N$ lattice.

\section{Future Work}
There are several ways to improve and extend our work. 
We have already done some characterizations of the subsets that maximize error and how the subsets transition from one configuration to another. We know we have a checkerboard configuration when $p=\left\lceil N^2/2 \right\rceil$ and we have a filled-in perimeter when $p=\sum_{i=1}^{m} 4(N-(2i-1))$ for different values of $m$. However the transition between these two configurations has still yet to be characterized.

Additionally, we hope to improve our lower bound work. Work can be done to find a characterization of the sets that minimize error. Furthermore, we hope to refine our lower bound formula by finding a rigorous lower bound the values of $p$ where $N^4/8p^2$ holds.

Finally, as Erd\H{o}s conjectured that all near-optimal sets have lattice structure, it is natural to extend results to other lattice structures. We expect that the error is similar.

\appendix
\section{Additional Calculations on the Lattice}\label{checkerboardcalculations}

As mentioned in Example \ref{ex:checkerboard}, our error calculations on the checkerboard lattice can be improved with an exact counting of its distances. First,  we may look at distances of the form $\sqrt{a^2}$, i.e., those for which $b=0$. We note these only appear on the checkerboard lattice when $a$ is even.

We count the total number of times that two points on the lattice are separated by a horizontal distance of $a$—this number matches the number of pairs which are separated by a vertical distance of $a$ by rotational symmetry. In the case where $N$ is even, each row of the lattice contains $N/2$ points, of which there are $N/2-a/2$ pairs at a distance $a$. We thus see that  
\begin{align}
	S_{a,0}& \ = \ 2N\left(  \frac{N}{2}-\frac{a}{2}\right) \nonumber \\
	& \ = \ N(N-a).
\end{align}

In the case where $N$ is odd, we assume that the checkerboard lattice is chosen in such a way that the four corners are included. In this case, we see that $(N+1)/2$ of the rows contain $(N+1)/2$ points, whereas $(N-1)/2$ rows contain $(N-1)/2$ points. Thus, 
\begin{align}
		S_{a,0}& \ = \ 2\left( \frac{N+1}{2}\left( \frac{N+1}{2}-\frac{a}{2} \right) +\frac{N-1}{2}\left( \frac{N-1}{2}-\frac{a}{2} \right)  \right)\nonumber \\
		& \ = \ N(N-a)+1.
\end{align}

We now look at distances of the form $\sqrt{2a^2}$, i.e., those formed by values $a  =  b$. We assume $N$ is odd.

As $N$ is odd, we notice that on the $N\times N$ checkerboard lattices all diagonals at $45^{\circ}$ contain an odd number of points. Moreover, for any $1\leq n\leq (N-1)/2$, there are four diagonals with $2n-1$ points; additionally, there are two diagonals with $N$ points. Finally, the distance $\sqrt{2a^2}=\sqrt{2}a$ appears on a diagonal with $2n-1$ points precisely $2n-1-a$ times; in the case where $n\leq a$, the distance is not present. 

In the case where $a$ is odd, we see $\sqrt{2}a$ is present on all diagonals with at least $a+2$ points. We have 
\begin{align}
S_{a,a}
&\ = \  4\left( \left( a+2 \right) -a+\left( a+4 \right) -a+\cdots+(N-2)-a \right)  +2(N-a) \nonumber \\
&\ = \  4\left(2+4+\cdots+N-a-2 \right)+2\left( N-a \right) \nonumber  \\
&\ = \  8\left( 1+\cdots+\frac{N-a-2}{2} \right) +2(N-a) \nonumber \\
&\ = \  8\left( \frac{\frac{N-a-2}{2}\frac{N-a}{2}}{2} \right) +2(N-a) \nonumber \\
&\ = \  (N-a-2)(N-a)+2(N-a)  \nonumber \\
&\ = \  (N-a)^2. 
\end{align}

In the case where $a$ is even, we see $\sqrt{2}a$ is present on all diagonals with at least $a+1$ points. Thus, 
\begin{align}
S_{a,a}
&\ = \  4\left( (a+1)-a+(a+3)-a+\cdots+(N-2)-a \right) +2(N-a)\nonumber \\
&\ = \  4\left( 1+3+\cdots+N-2-a \right) +2\left( N-a \right)  \nonumber\\
&\ = \  4\left( 2+4+\cdots+N-1-a \right) -4\left( \frac{N-a-1}{2} \right)+2(N-a) \nonumber \\
& \ = \ 8\left( \frac{\frac{N-a-1}{2}\frac{N-a+1}{2}}{2} \right) -2(N-a-1)+2(N-a) \nonumber\\
&\ = \  (N-a-1)(N-a+1)+2 \nonumber\\
&\ = \  (N-a)^2-1+2 \nonumber \\
&\ = \  (N-a)^2+1. 
\end{align}

In the case where $N$ is even, the calculation is slightly more complicated, as the only possible arrangement of $N ^2/2$ points in a checkerboard pattern results in a loss of 4-fold rotational symmetry. Without loss of generality, we see that the bottom-left to top-right diagonals each contain an even number of points, whereas the top-left to bottom-right diagonals each contain an odd number of points. In particular, the cardinalities of the bottom-left to top-right diagonals are \begin{equation}
2,\; 4,\; 6,\ldots , N-2,\; N,\; N-2,\ldots , 6,\; 4,\; 2
.\end{equation}
The top-left to bottom-right diagonals instead have cardinalities
\begin{equation}
1,\; 3,\; 5,\ldots, N-3,\;  N-1,\; N-1,\; N-3,\; ,\ldots , 5,\; 3,\; 1
.\end{equation}
For $a$ even, the distance $\sqrt{2}a$ is on any diagonal with cardinality $n>a$, with frequency $n-a$. Thus, we calculate
\begin{align}
S_{a,a}
 &\ = \  2\sum_{n =  a/2}^{(N-2)/2}(2n+1-a)+2\sum_{n = 1+a/2}^{(N-2)/2}\left( 2n-a \right)  +(N-a)  \nonumber\\
 &\ = \ 2\left( 1+3+\cdots+N-1-a \right) +2\left( 2+4+\cdots+N-2-a \right) +N-a \nonumber\\
 &\ = \ 2\left( \frac{\left( N-a-1 \right) \left( N-a \right) }{2} \right) +N-a \nonumber\\
 &\ = \ \left( N-a-1 \right) \left( N-a \right) +N-a \nonumber \\
 &\ = \ \left( N-a \right) ^2. 
\end{align}
For $a$ odd, the calculation is nearly identical, although $a+1$ becomes the first even diagonal cardinality on which a distance $a$ appears, and $a+2$ the first odd diagonal cardinality. Hence, we see
\begin{align}
S_{a,a}
 &\ = \ 2\sum_{n=(a+1)/2}^{(N-2)/2}(2n+1-a)+2\sum_{n=(a+1)/2}^{(N-2)/2}\left( 2n-a \right)  +(N-a)  \nonumber\\
 &\ = \  2\left( 2+4+\cdots+N-1-a \right)+2\left( 1+3+\cdots+N-2-a \right) +N-a \nonumber\\
 &\ = \ 2\left( \frac{\left( N-a-1 \right) \left( N-a \right) }{2} \right) +N-a\nonumber \\
 &\ = \ \left( N-a-1 \right) \left( N-a \right) +N-a \nonumber\\
 &\ = \ \left( N-a \right) ^2. 
\end{align}

\end{document}